\documentclass[11pt]{amsart}
\usepackage[utf8]{inputenc}
\usepackage{amsmath,amsthm,comment}
\usepackage{amsfonts}
\usepackage{amssymb}

\usepackage{graphicx,xy,pstricks}
\input xy
\xyoption{all}
\newcommand{\executeiffilenewer}[3]{%
 \ifnum\pdfstrcmp{\pdffilemoddate{#1}}%
 {\pdffilemoddate{#2}}>0%
 {\immediate\write18{#3}}\fi%
}

\theoremstyle{definition} 

 \newtheorem{definition}{Definition}[section]
 \newtheorem{remark}[definition]{Remark}

\theoremstyle{plain}      

 \newtheorem{proposition}[definition]{Proposition}
 \newtheorem{theorem}[definition]{Theorem}
 \newtheorem{corollary}[definition]{Corollary}
 \newtheorem{lemma}[definition]{Lemma}

\newtheorem*{theorem*}{Theorem}

\newcommand{\C}{\mathbb{C}}
\newcommand{\Q}{\mathbb{Q}}
\newcommand{\N}{\mathbb{N}}

\newcommand{\Z}{\mathbb{Z}}

\renewcommand{\S}{\mathcal{S}}
\renewcommand{\O}{\mathcal{O}}
\renewcommand{\o}{{\!\scriptscriptstyle\mathcal{O}}}

\newcommand{\Sred}{\mathcal{S}^r}

\DeclareMathOperator{\tr}{Tr}

\DeclareMathOperator{\mcg}{Mod}
\DeclareMathOperator{\mcgu}{\widetilde{Mod}}

\DeclareMathOperator{\SL}{SL}
\DeclareMathOperator{\GL}{GL}
\DeclareMathOperator{\PGL}{PGL}

\DeclareMathOperator{\en}{End}

\DeclareMathOperator{\im}{Im}
\DeclareMathOperator{\id}{Id}
\DeclareMathOperator{\coker}{Coker}

\title{Introduction to quantum representations of mapping class groups}
\author{Julien Marché}
\address{Sorbonne Universit\'e, IMJ-PRG, 75252 Paris c\'edex 05, France}
\email{julien.marche@imj-prg.fr}
\begin{document}
\date{}
\maketitle

\begin{abstract}
We provide an (almost) self-contained construction of the Witten-Reshetikhin-Turaev representations of the mapping class group. We describe its properties including its Hermitian structure, irreducibility and integrality (at prime level). The construction of these notes relies only on skein theory (Kauffman Bracket) and does not use surgery techniques. We hope that they will be accessible to non-specialists. 
\end{abstract}
\section{Introduction}
The aim of these notes is to give a short and self-contained construction of the quantum representations of the mapping class group of a surface. These representations were discovered in the early 90's as a byproduct of a more general structure called topological quantum field theory (TQFT) that will not be covered in these notes as will not be covered the various relations of the quantum representations with arithmetic groups, semi-classical analysis, non-abelian Hodge theory, etc. Although many constructions are available now, none is easily accessible to non-specialists: most of them use surgery techniques and a modular category based either on the Kauffman bracket skein module or on the representation theory of quantum groups. 
Here we provide a new construction based on old ideas. Formally, we will construct a finite dimensional projective representation of the mapping class group over the cyclotomic field of order $4r$. This corresponds to the so-called SU$_2$ TQFT of level $r-2$, first introduced by Witten and Reshetikhin-Turaev (\cite{w,rt}). Our approach is in the spirit of \cite{bhmv} and \cite{skeinroberts} in that it uses skein theory of the Kauffman bracket. It is simpler however in that it uses only basics properties of the Jones-Wenzl idempotents, and no surgery techniques.  The main tool is the notion of reduced skein module, studied independently by Roberts and Sikora, see \cite{sikora}. 

Let $K$ be the cyclotomic field of order $4r$. We define the reduced skein module of a 3-manifold $M$ - denoted by $\Sred(M)$ - as the $K$-vector space generated by banded links in $M$ modulo three relations, the first two are the usual Kauffman relations, the third one involves a more complicated linear combination of banded links encoded by the Jones-Wenzl idempotent $f_{r-1}$. These modules are compatible with gluing, meaning that if $M$ and $N$ are two three manifolds that we glue along a part of their boundary, the disjoint union operation induces a bilinear map $\Sred(M)\times\Sred(N)\to \Sred(M\cup N)$. 

When $M=\Sigma\times[0,1]$, the gluing operations endow $\Sred(M)$ - which we denote from now on by $\Sred(\Sigma)$ - with the structure of a $K$-algebra. The main result is the following:

\begin{theorem*}There exists $n\in \N$ such that the algebra $\Sred(\Sigma)$ is isomorphic to the matrix algebra $M_n(K)$. 
\end{theorem*}



This gives a construction of the quantum representation by the following argument. Let $\mcg(\Sigma)$ be the mapping class group of $\Sigma$. Any element in $\mcg(\Sigma)$ is represented by a diffeomorphism $f$ which acts on $\Sigma\times[0,1]$ by $(x,t)\mapsto(f(x),t)$ and hence on $\Sred(\Sigma)$ by algebra automorphism. By the Skolem-Noether theorem, this automorphism, viewed as an automorphism of $M_n(K)$ is a conjugation by some uniquely defined element $\rho(f) \in\PGL_n(K)$. The mapping $f\mapsto \rho(f)$ is the aforementioned quantum projective representation: $$\rho:\mcg(\Sigma)\to \PGL_n(K).$$

To go further, consider the diffeomorphism of $\Sigma\times[0,1]$ given by $(x,t)\mapsto (x,1-t)$. It is an involution reversing the orientation and commuting with the action of $\mcg(\Sigma)$. This induces an anti-involution of $\Sred(\Sigma)$ which is antilinear with respect to the involution of $K$ mapping $A$ to $A^{-1}$. It is well-known that such involutions in matrix algebras correspond to Hermitian forms, see for instance \cite{involution}, p.1. Hence, we find that there exists a Hermitian form $h$ on $K^n$ such that $\rho$ takes values in $PU(h)$.

We will prove the main theorem by observing that if $\Sigma$ is the boundary of a handlebody $H$, then $\Sred(\Sigma)$ acts on $\Sred(H)$ by gluing $\Sigma\times[0,1]$ to the boundary of $H$. This defines a morphism of algebras $\Sred(\Sigma)\to \en(\Sred(H))$ and the theorem follows once we prove that it is an isomorphism and that $\Sred(H)$ is finite dimensional. In particular the latter space is naturally the space on which the quantum representation acts projectively. Given a banded trivalent graph $\Gamma$ embedded in $H$ such that $H$ retracts on $\Gamma$, we will construct elements $\Gamma_c\in \Sred(H)$ for every map $c:E(\Gamma)\to \{0,\ldots,r-2\}$ which is $r$-admissible, that is which satisfies the following relations for every triple of edges $e,e',e''\in E(\Gamma)$ incident to a vertex:

$$c(e)+c(e')+c(e'')\in \{0,2,\ldots, 2r-4\},\quad c(e)\le c(e')+c(e'').$$

We will show that these elements form an orthogonal basis of $\Sred(H)$ and compute their norm. This will give the dimension of the representation (known as the Verlinde formula) and the signature of the Hermitian form given an embedding of $K$ in $\C$. 

We will end the notes with some explicit formulas for the quantum representations and a proof that these representations are irreducible when $r$ is an odd prime. We will see that in the latter case, the reduced skein modules $\Sred(M)$ can be defined over the ring of integers $\O\subset K$. Denoting by $\Sred_\o(M)$ this integral version, we find that the mapping class group $\mcg(\Sigma)$ preserves $\Sred_\o(\Sigma)\subset\Sred(\Sigma)$. The space $\Sred_\o(\Sigma)$ is an {\it order}, meaning that it is both a sub-algebra and a finitely generated sub-$\O$-module which generates $\Sred(\Sigma)$ over $K$. This implies that the quantum representation stabilizes a $\O$-lattice in $\Sred(H)$. In particular, its image lies in an arithmetic group.
We then give indications on how one can introduce marked points in this settings.

{\bf Acknowledgments:} I would like to thank B. Deroin, P. Eyssidieux, B. Klingler, M. Maculan, G. Masbaum (in particular), A. Sambarino, N. Tholozan and M. Wolff for their encouragements and/or help during the writing of these notes.

\section{Skein Modules}
\subsection{Definition}
Let $M$ be a compact oriented 3-manifold with boundary and $P$ be a collection of disjoint oriented arcs embedded in $\partial M$ ($P$ and $\partial M$ may be empty).
\begin{definition}
A banded trivalent graph in $M$ with boundary $P$ is a (possibly empty) pair $(\Gamma,S_\Gamma)$ where 
\begin{itemize}
\item[-] $S_\Gamma$ is an oriented surface with boundary embedded in $M$ such that $S_\Gamma\cap \partial M=P$. We assume that the intersection is transversal and that the orientations of $\partial S_\Gamma$ and $P$ coïncide at the intersection.
\item[-] $\Gamma$ is a graph with vertices of valency 1 or 3. The set of univalent vertices is denoted by $\partial \Gamma$. $\Gamma$ can contain circles, loops and multiple edges. 
\item[-] There is an embedding $\Gamma\subset S_\Gamma$ such that $\Gamma\cap \partial S_\Gamma=\partial \Gamma \subset P$ and $S_\Gamma$ retracts on $\Gamma$ by deformation preserving $P$. 
\end{itemize}
\end{definition}

\begin{figure}[htbp]
\begin{center}
 \def\svgwidth{8cm}
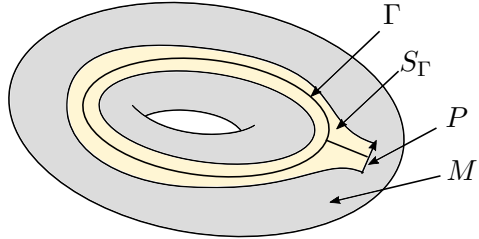
\caption{An example of banded trivalent graph}\label{graph}
\end{center}
\end{figure}

In the sequel we will remove $S_\Gamma$ from the notation although it will be always present. Also, in the next figures, we only show $\Gamma$. The surface will be understood to be a tubular neighbourhood of the graph in the plane were it is drawn. In particular, the system of arcs $P$ will appear as a system of points.
If $\Gamma$ has no trivalent vertices we will say that $\Gamma$ is a banded tangle and denote it preferably by $L$. If $\Gamma$ has no vertices at all (hence $P=\emptyset$ and $\Gamma$ is a union of circles), $\Gamma$ will be called a banded link.

Let $R$ be a commutative ring with unit and $A\in R^\times$ be an invertible element. We will denote by $\hat{\S}(M)$ the free $R$-module generated by isotopy classes of banded tangles with boundary $P$. This $R$-module comes with a simple gluing operation described below.

Let $M$ and $N$ be two 3-manifolds with $P_M\subset \partial M$ and $P_N\subset \partial N$ two system of arcs.  Let $\Sigma$ be a surface with two embeddings $i_M:\Sigma\to \partial M$ and $i_N:\Sigma\to \partial N$ respectively preserving and reversing the orientation. We suppose that the arcs are compatible in the sense that $i_M^{-1}(P_M)=i_N^{-1}(P_N)$. Then, mapping $(\Gamma,S_\Gamma),(\Gamma',S_{\Gamma'})$ to $(\Gamma\cup\Gamma',S_\Gamma\cup S_{\Gamma'})$ induces a $R$-bilinear map

\begin{equation}\label{gluing}
\hat{\S}(M,P_M)\times \hat{\S}(N,P_N)\to \hat{\S}(M\cup_\Sigma N,P_M'\cup P_N')
\end{equation}
where $P_M'$ and $P_N'$ denote the remaining system of arcs. We will denote by $\langle \cdot,\cdot\rangle_{M,N}$ this map. 

Let $M$ be a 3-manifold and $P\subset \partial M$ be a system of arcs. We will consider two kinds of elements of $\hat{\S}(M,P)$.
\begin{itemize}
\item[-] For the first kind, we consider a ball $B^3$ embedded in $M$ with $Q\subset \partial B^3$ a system of four arcs and we let $L_\times,L_0,L_\infty\in \hat{\S}(B^3,Q)$ be the elements showed in the left of Figure \ref{Kauffman}. Then for any $x\in \hat{\S}(M\setminus B^3,Q)$ we consider the linear combination 
$$\langle L_\times-AL_0-A^{-1}L_\infty,x\rangle_{B^3,M\setminus B^3}\in \hat{\S}(M,P)$$

\item[-] For the second kind, we consider a ball $B^3$ embedded in $M$ with no arcs and the banded links $L_U$ and $L_\emptyset$ showed in the right of Figure \ref{Kauffman}. For any $x\in \hat{\S}(M\setminus B^3,P)$ we consider the linear combination

$$\langle L_U+(A^2+A^{-2})L_\emptyset,x\rangle_{B^3,M\setminus B^3}\in \hat{\S}(M,P).$$
\end{itemize}

\begin{figure}[htbp]
\begin{center}
 \def\svgwidth{12cm}
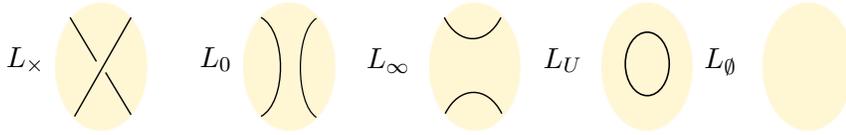
\caption{Kauffman relations}\label{Kauffman}
\end{center}
\end{figure}

\begin{definition}
The Kauffman skein module $\S(M,P)$ is by definition the quotient of $\hat{\S}(M,P)$ by the submodule generated by the above elements called Kauffman relations of the first and second kind respectively. 
\end{definition}

A key property of this definition is that all relations have the form $\langle r,x\rangle=0$ where $r$ is in a ball $B^3$ in the interior of $M$ and $x$ is in $M\setminus B^3$. If $M$ is glued to another manifold $N$, $B^3$ is still a ball in $M\cup N$ and $x$ may be viewed in $M\cup N\setminus B^3$: it follows that the gluing operations \eqref{gluing} are still well-defined and will be denoted in the same way. 
Skein modules were introduced independently by Przytycki and Turaev. We refer to \cite{prz} for details of the proof sketched in this section. 

\subsection{Skein module of thickened surfaces}

At this point, it is not clear whether the skein module is a manageable $R$-module. In this section, we will sketch the proof that when $M=\Sigma\times[0,1]$ for some compact oriented surface $\Sigma$, the module $\S(M)$ is free, generated by banded links embedded in $\Sigma\times\{1/2\}$. Here comes the precise statement. 

Let $\Sigma$ be a surface and $P\subset \partial \Sigma$ be a system of arcs. The pair $(\Sigma,P)$ will also denote the pair $\Sigma\times [0,1],P\times \{\frac 1 2\}$ to simplify the notation. A {\it simple tangle} $L$ is by definition a 1-manifold $L\subset \Sigma$ such that no component of $L$ bounds a disc embedded in $\Sigma$. The surface $S_L$ is a tubular neighbourhood of $L$ in $\Sigma$ such that $S_L\cap \partial \Sigma=P$. 

\begin{theorem}\label{structure_skein}
For any pair $(\Sigma,P)$ as above, the skein module $\S(\Sigma,P)$ is freely generated by isotopy classes of simple tangles. 
\end{theorem}

\begin{proof}
The proof is very instructive but too long to be given in detail. It relies on (a version of) the Reidemeister theorem which states that any banded tangle can be represented by a diagram in $\Sigma$, that is a submanifold $L\subset \Sigma$ with simple crossings and - at each crossing - the information of which branch goes above. By taking a tubular neighborhood of $L$ in $\Sigma$ and separating the branches at each crossing as indicated, one gets indeed a banded tangle. 
The second part of the Reidemeister theorem is that two banded tangles associated to the diagrams $L$ and $L'$ are isotopic if and only if they are related by a sequence of one of the three moves shown at the top of Figure \ref{reidemeister}. 

\begin{figure}[htbp]
\begin{center}
 \def\svgwidth{12cm}
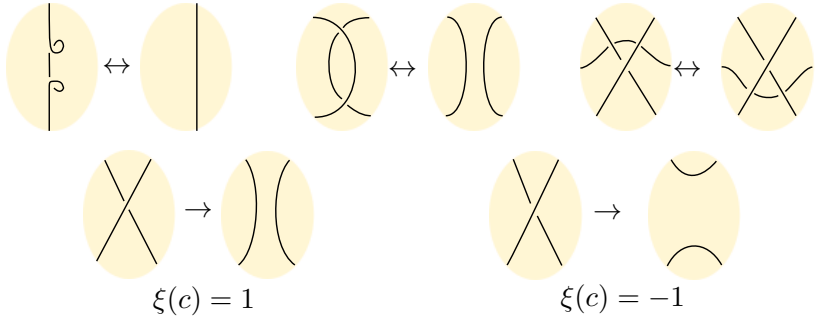
\caption{Reidemeister moves}\label{reidemeister}
\end{center}
\end{figure}

Consider then a diagram $L\subset \Sigma$ and a map $\xi:c(L)\to \{\pm 1\}$ where $c(L)$ denotes the set of crossings of $L$. We denote by $L_\xi$ the banded tangle obtained by smoothing the crossings of $L$ as showed at the bottom of Figure \ref{reidemeister}. We also denote by $n(L_\xi)$ the number of components of $L_\xi$ bounding a disc and by $L_\xi'$ the sub-tangle obtained by removing those trivial components. 
Then, the proof of the theorem consists in showing that the map $\Phi: \hat{S}(\Sigma,P)\to \bigoplus\limits_{L\text{ simple}} R L$ given by 
$$\Phi(L)=\sum_{\xi:c(L)\to\{\pm 1\}} A^{\sum_{c\in c(L)}\xi(c)}(-A^2-A^{-2})^{n(L)} L_\xi'$$
is well-defined, that is invariant by Reidemeister moves and sends a Kauffman relation to $0$. Both are simple to show, we refer to \cite{prz} for details.  
\end{proof}

We end this section by the following observation: suppose that $P=\emptyset$ and take $L_1,L_2$ two banded links in $\Sigma\times [0,1]$. By shrinking $[0,1]$ into $[0,1/2]$ for $L_1$ and into $[1/2,1]$ for $L_2$, one can consider their disjoint union $L_1\cup L_2\in \Sigma\times[0,1]$. This is again a gluing operation which induces an structure of associative algebra on $\S(\Sigma)$. Its unit is given by the empty link.

\subsection{Jones-Wenzl idempotents}\label{sectionjw}
For any $n\in \N$, we fix a standard collection $P_n$ of arcs of cardinality $n$ in $(0,1)\times\{1/2\}$ and set $$T_n=\S([0,1]^3, P_n\times\{0,1\}).$$
If $x,y\in T_n$, we can view $x$ in $\S([0,1]^2\times[0,1/2],P_n\times\{0,1/2\})$ and $y$ in $\S([0,1]^2\times[1/2,1],P_n\times\{1/2,1\})$.
Gluing them endows $T_n$ with the structure of an algebra called the Temperley-Lieb algebra. Its unit is the class of $P_n\times [0,1]$ that we denote by $1_n$. 
We further define $e_i\in T_n$ for $i=1,\ldots, n-1$ as in Figure \ref{e_i}. 

\begin{figure}[htbp]
\begin{center}
 \def\svgwidth{10cm}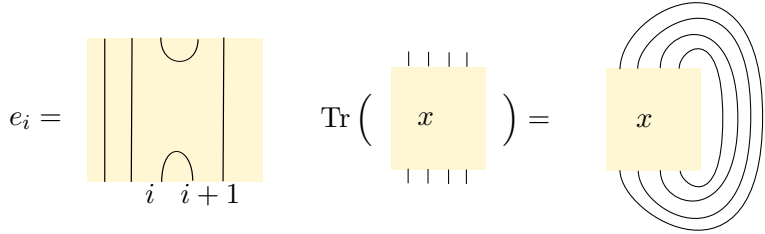
\caption{Generators and trace of the Temperley-Lieb algebra}\label{e_i}
\end{center}
\end{figure}

Let $I_n\subset T_n$ be the submodule generated by tangles which intersect (up to isotopy) the middle square $[0,1]^2\times\{1/2\}$ in strictly less than $n$ points. It is clearly a two-sided ideal containing the $e_i$. 

Finally, we observe that the unique simple tangle in $T_n$ which is not in $I_n$ is the trivial one. Hence, $T_n/I_n\simeq R 1_n$ and we denote by $\epsilon:T_n\to R$ the unique map such that $x=\epsilon(x)1_n\mod I_n$. The Jones-Wenzl idempotents appear when one tries to decompose $T_n$ as a product algebra $I_n\times R$. 

To that aim, we observe that the juxtaposition of tangles gives a morphism $T_n\otimes T_m\to T_{n+m}$. We will identify $x\in T_n$ with $x\otimes 1_1\in T_{n+1}$ . It consists in adding a trivial strand on the right of a diagram. We also define a trace $\tr:T_n\to R$ by ``closing" the tangle in a standard way in $B^3$ as in the right hand side of Figure \ref{e_i} (by identifying $\S(B^3)$ with $R$). We also set $[n]=\frac{A^{2n}-A^{-2n}}{A^2-A^{-2}}$.

\begin{theorem}\label{jw}
Let $r\in \N$ be such that $[1],\ldots, [r-1]$ are invertible in $R$ and set $f_0=1_0$ and for all $0<n<r$ (setting $e_{-1}=0$)
$$f_{n}=f_{n-1}+\frac{[n-1]}{[n]}f_{n-1} e_{n-1} f_{n-1} \in T_{n}$$
These elements satisfy the following properties:
\begin{enumerate}
\item $f_n^2=f_n$.
\item $f_n e_i=e_if_n=0$ for all $i<n$. 
\item $f_n-1_n\in I_n$.
\item $\tr f_n=(-1)^n[n+1]$.
\end{enumerate}
\end{theorem}
We skip the proof of the three first items which follows easily by induction (see \cite{Lickorish}, Chap.13) and postpone the last item to the end of this section.

The main property of the idempotents is that $xf_n=f_nx=0$ whenever $x\in I_n$ since one can show that $I_n$ is generated by the $e_i$ as an algebra (without unit). This implies the following equation: $f_nx=xf_n=\epsilon(x)f_n$. In practice, $\epsilon(x)$ can be computed by solving the crossings of $x$ in all possible ways that avoid back-tracking. 
The following examples will be useful.

\begin{lemma}\label{cerclage}
Let $x_n, y_n,z_n$ be the elements of $T_n$ showed in Figure \ref{cercle}. Then $\epsilon(x_n)=-\frac{[n+2]}{[n+1]}$, $\epsilon(y_n)=-A^{2n+2}-A^{-2n-2}$ and $\epsilon(z_n)=(-1)^nA^{n(n+2)}$. 
\end{lemma}
\begin{proof}
By considering the recursive definition of $f_{n+1}$, we get immediately $\epsilon(x_n)=-[2]+\frac{[n]}{[n+1]}$. From the formula $[n+1][m+1]=[n][m]+[n+m+1]$ we get the result. We refer to \cite{Lickorish}, Lemma 14.1 and Lemma 14.2 for the last two. 
\end{proof}

\begin{figure}[htbp]
\begin{center}
 \def\svgwidth{12cm}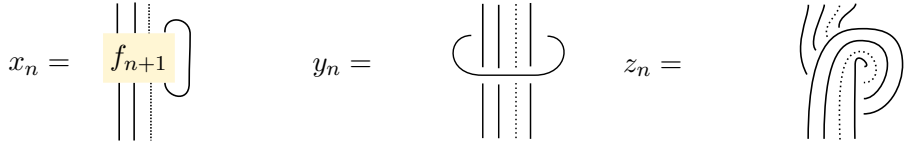
\caption{Partial trace, embracing and twisting}\label{cercle}
\end{center}
\end{figure}

Let $S$ be the annulus obtained by gluing the opposite sides of $[0,1]\times\{1/2\}\times [0,1]$. Thanks to Theorem \ref{structure_skein}, the algebra $\Sred(S)$ is isomorphic to $R[z]$ where $z$ denotes the core of the annulus $S$, formally $[1/3,2/3]\times\{1/2\}\times [0,1]$. 

Let us recall the definition of the Tchebychev polynomials of the second kind. They satisfy $S_0=1,S_1=z$ and $S_{n+1}=zS_n-S_{n-1}$ for all $n>0$ and the identity $S_n(-y-y^{-1})=(-1)^n\frac{y^{n+1}-y^{-n-1}}{y-y^{-1}}$.

\begin{lemma}\label{tchebychev}
Let $\hat{f}_n$ be the closure of the idempotent $f_n$ obtained by identifying the opposite sides of the cube. Then $\hat{f}_n=S_n(z)$.
\end{lemma}
\begin{proof}
The recursion formula for $f_n$ implies $\hat{f}_{n+1}=z\hat{f}_n+\frac{[n]}{[n+1]} \hat{x}_{n-1}$ using the notation of Lemma \ref{cerclage}. From the equality $f_nf_{n+1}=f_{n+1}$ and Lemma \ref{cerclage}, we get $\hat{x}_{n-1}=-\frac{[n+1]}{[n]}\hat{f}_{n-1}$. This shows that $\hat{f}_n$ satisfies the same recursion relation as $S_n$. As $\hat{f}_0=1$ and $\hat{f}_1=z$, the conclusion follows.  
\end{proof}

This lemma proves the last item of Theorem \ref{jw} by the following observation. The standard embedding of $S$ in $B^3$ induces a map $i:\S(S)\to \S(B^3)=R$. As  $z^n$ is a collection of disjoint trivial circles, we have $i(z^n)=(-A^2-A^{-2})^n$. Hence $\tr(f_n)=i(\hat{f}_n)=S_n(-A^2-A^{-2})=(-1)^n[n+1]$. 

\section{The $r$-reduced skein module}

We fix $r\ge 2$ and suppose in this section that $R$ is equal to the cyclotomic field $K=\Q[A]/(\phi_{4r}(A))$ where $\phi_{4r}$ denotes the cyclotomic polynomial of order $4r$. This implies that $[n]\ne 0$ for all $n<r$ and $[r]=0$. In particular, only the idempotents $f_0,\ldots,f_{r-1}$ exist. The last one has the following crucial vanishing property:
\begin{lemma}\label{vanishing}
For all $x\in T_{r-1}$, $\tr (f_{r-1}x)=0$. 
\end{lemma}
\begin{proof}
One write $x=1_{r-1}+y$ with $y\in I_{r-1}$. By Property $2$ of Theorem \ref{jw}, we have $ f_{r-1}y=0$ and by Property $4$, $\tr f_{r-1}=(-1)^{r-1}[r]=0$, hence the result. 
\end{proof}

This suggests the following definition for any $3$-manifold $M$ with system of arcs $P$. 
\begin{definition} For any embedding of $[0,1]^3\to M$, and any $x\in \S(M\setminus [0,1]^3, P_{r-1}\times\{0,1\} \cup P)$ we call $(r-1)$-relation the following element: 
$$\langle f_{r-1},x\rangle_{[0,1]^3,M\setminus[0,1]^3}\in \S(M,P).$$

We define $\Sred(M,P)$ to be the quotient of $\S(M,P)$ by the $K$-subspace generated by the $(r-1)$-relations.
\end{definition}

As for the standard skein module, the relations being local, the reduced skein module is still compatible with gluing. Moreover, thanks to Lemma \ref{vanishing}, the $(r-1)$-relation hold in $B^3$, or in other terms we still have $\Sred(B^3)\simeq K$. 

\begin{remark}
Let $\overline{M}$ be the manifold $M$ with opposite orientation. The image of a Kauffman relation in $M$ is no longer a Kauffman relation in $\overline{M}$ unless we exchange $A$ and $A^{-1}$. Formally, there is an isomorphism $\S(M)\simeq\S(\overline{M})$ which maps $\lambda L$ to $\overline{\lambda}L$ where the involution $\lambda\mapsto\overline{\lambda}$ of $K$ is defined by $\overline{A}=A^{-1}$. We observe that the idempotents are constructed recursively with quantum integers $[n]$ which are fixed by the involution. This shows that the same map induces an anti-linear isomorphism $\Sred(M)\simeq \Sred(\overline{M})$. 
\end{remark}
\subsection{The reduced skein module of a handlebody}

Let $H$ be a handlebody, that is a $3$-manifold which retracts by deformation to a graph. Our first task is to find a finite generating set for $\Sred(H)$. 

We observe first that there exists a trivalent banded graph $(\Gamma,S_\Gamma)$ embedded in $H$ such that $H$ is homeomorphic to $S_\Gamma\times [0,1]$. 
Consider for any edge $e$ of $\Gamma$ a disc $D_e$ embedded in $H$ satisfying $\partial D_e\subset \partial H$ and intersecting $\Gamma$ in one point: we will say that $D_e$ is dual to $e$. One can suppose that these discs are disjoint and that the complement of their union is a collection of balls. 

\begin{lemma}\label{generateurs}
The vector space $\Sred(H)$ is generated by banded links which intersect each dual disc in at most $r-2$ points. 
\end{lemma}
\begin{proof}
We proceed by induction: any time we find $r-1$ points on $L\cap D_e$ for some edge $e$, we can find up to isotopy a cube $[0,1]^3$ embedded in $D_e\times[0,1]$ such that $[0,1]^3\cap L=1_{r-1}$. By writing $f_{r-1}=1_{r-1}+x_{r-1}$ with $x_{r-1}\in I_{r-1}$, one can replace modulo a $(r-1)$-relation $1_{r-1}$ by $-x_{r-1}$ which has strictly less crossing points with $D_e$. The result follows. 
\end{proof}

\begin{definition}
A $r$-coloring of a trivalent graph $\Gamma$ is a map $c:E(\Gamma)\to \{0,\ldots,r-2\}$ satisfying for all edges $e,e',e''$ incident to a vertex $v$ the following  triangular conditions:
$$c(e)+c(e')+c(e'')\text{ is even}\quad\text{ and }\quad c(e)\le c(e')+c(e'').$$
These properties ensure that at each vertex there exists $i,i',i''\in \N$ such that $c(e)=i'+i'', c(e')=i+i''$ and $c(e'')=i+i'$. These integers are called the {\it internal colors} at the vertex $v$. 

We can associate to each coloring $c$ the {\it skein expansion} $\Gamma_c\in \Sred(S_\Gamma)$ by the following gluing procedure. Put $f_{c(e)}$ along each edge $e$ and join the remaining strands around the vertices in the unique way that avoid crossing and back-tracking. For instance, at a vertex where $e,e',e''$ are incident,  $i''$ strands will go from $e$ to $e'$, etc. as in Figure \ref{triad}. 
\end{definition}
\begin{figure}[htbp]
\begin{center}
 \def\svgwidth{6cm}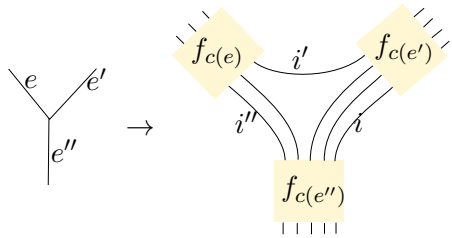
\caption{Skein expansion around a vertex}\label{triad}
\end{center}
\end{figure}

\begin{lemma}\label{reduction}
The elements $\Gamma_c$ where $c$ runs over $r$-colorings generate $\Sred(S_\Gamma)$.
\end{lemma}

\begin{proof}
We proceed again by induction. By Lemma \ref{generateurs}, $\Sred(S_\Gamma)$ is generated by simple links which cross each dual disc $D_e$ at $n(e)$ points where $n(e)<r-1$. Inserting a cube, one can write $1_{n(e)}=f_{n(e)}-x_{n(e)}$. If we replace $1_{n(e)}$ by $f_{n(e)}$, we observe that the only way for joining the strands around a vertex which does not vanish is the one shown in Figure \ref{triad}. Hence the element we obtained is proportional to some $\Gamma_c$. By induction, as the element obtained by inserting $x_{n(e)}$ has strictly less crossing points with $D_e$, it can be expressed by skein expansions of colored graphs. 
\end{proof}

We will even reduce this generating family by using the identities shown in Figure \ref{fusion}. All these identities are understood in $\S(B^3,P)$ where $P$ is a fixed collection of points with a cardinality coherent with the diagram ($2n+2$ in the first one, for instance). If no color is indicated, it is understood to be 1. 
The first line comes from the definition of the skein expansion of colored graphs and the recursion formula for the Jones-Wenzl idempotents. The last ones can be derived by induction. We refer to \cite{MasbaumVogel} for a proof. 

\begin{figure}[htbp]
\begin{center}
 \def\svgwidth{10cm}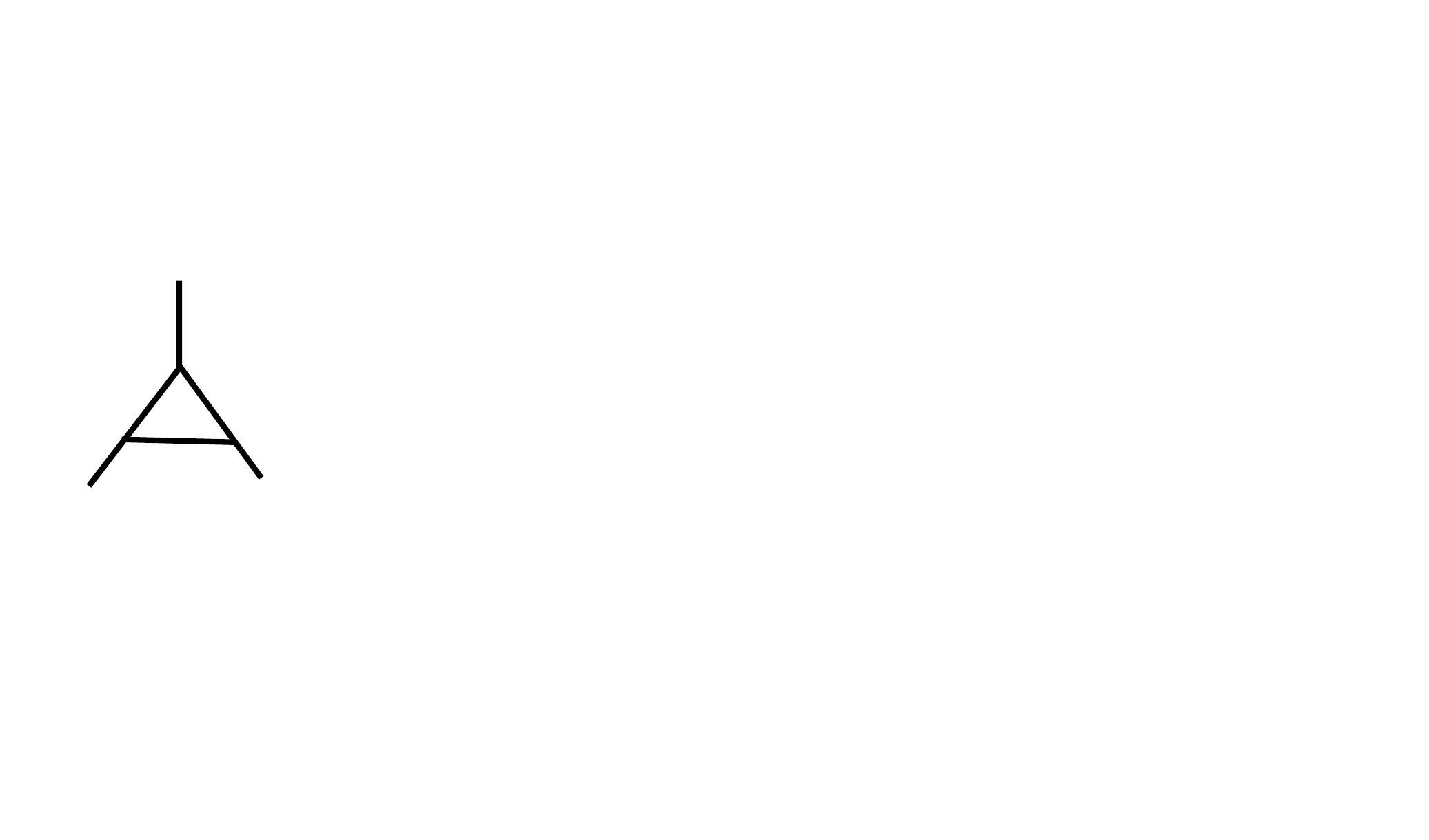
\caption{Simple fusion rules}\label{fusion}
\end{center}
\end{figure}

\begin{lemma}
Let $a,b,c\in \{0,\ldots,r-2\}$ be a triangular triple satisfying $a+b+c\ge 2r-2$. A $Y$-shaped trivalent graph colored with $(a,b,c)$ vanishes in the $r$-reduced relative skein module of the ball. 
\end{lemma}
\begin{proof}
Let us call this $Y$-shaped graph a $(a,b,c)$-triad. 
Using the second line of Figure \ref{fusion}, we see that the vanishing of an $(a,b,c)$-triad implies the vanishing of the $(a,b-1,c+1)$-triad provided one has $a+b-c<2r$ and $b<r$. As we may write $a=j+k,b=i+k,c=i+j$, these conditions read $i+k<r$ and $k<r$ and hence are automatically satisfied. 

Suppose that we have $a\le b\le c$: if $a=0$, then $b=c=r-1$ and the triad vanishes as it contains a $(r-1)$-relation. If $c=r-1$, we stop for the same reason. Else, we replace the triad $(a,b,c)$ with $(a-1,b,c+1)$ and use induction. 
\end{proof}

\begin{definition} 
We will say that a $r$-coloring $c$ is $r$-admissible if for any triple $(e,e',e'')$ of edges incident to a same vertex one has $c(e)+c(e')+c(e'')<2r-2$. 
\end{definition}
The lemmas of these sections imply that the family $(\Gamma_c)$ for $c$ a $r$-admissible coloring is a generating set for $\Sred(H)$. We will show in the next section that they form a basis. 

\subsection{A Hermitian form on $\Sred(H)$}
Let us consider a standard embedding of $[0,1]^2$ into $S^2$ and take a point in the complement called $\infty$. We define the spherical reduced Temperley-Lieb algebra $TS_n=\Sred(S^2\times [0,1],P_n\times\{0,1\})$ and observe that this induces a surjection $\pi:T_n\to TS_n$ as any tangle in $S^2\times [0,1]$ avoid $\{\infty\}\times [0,1]$ up to isotopy.  

\begin{lemma}\label{annulationspherique}
For any $n\in 2\N$, there is an element $P_n=\sum_i \lambda_i L_i\in T_n$ where each tangle $L_i$ does not intersect $[0,1]^2\times \{1/2\}$ such that $\pi(P_n-1_n)=0$. If $n$ is odd, we set $P_n=0$ and the same equation holds. 
\end{lemma}
\begin{proof}
Let us prove the lemma by a double induction by setting $P_0=1_0$. Take $n$ satisfying $0<n<r-1$. One can write $1_n=f_n+z_n$ where $z_n\in I_n$. As $z_n$ can be written as a sum of tangles with less than $n$ crossings with $[0,1]^2\times\{1/2\}$, we can apply the induction hypothesis and reduce the problem to $f_n$. 
The key point is that the tangle $x_n$ of Lemma \ref{cerclage} is isotopic in $S^2\times [0,1]$ to the disjoint union of $1_n$ with a trivial circle. Using the corresponding formula we get in $TS_n$ $(-A^{2n+2}-A^{-2n-2})f_n=(-A^2-A^{-2})f_n$. This implies that $f_n=0$ thanks to our assumptions on $n$. In particular, we have $\pi(f_1)=\pi(1_1)=0$ which allows to start the double induction.
Suppose now that $n\ge r-1$. We can as in Lemma \ref{reduction} join $(r-1)$ parallel strands and use the vanishing of $f_{r-1}$ to reduce the number of intersection points and apply the induction hypothesis.
\end{proof}

\begin{corollary}
For any $g\ge 1$, let $M_g=(S^2\times S^1)^{\# g}$ be the connected sum of $g$ copies of $S^2\times S^1$. Then $\Sred(M_g)$ is naturally isomorphic to $K$. 
\end{corollary}
\begin{proof}
Consider the map $K\to \Sred(M_g)$ mapping $1$ to the empty link. We wish to construct an inverse. To that aim, let $S^2\subset M_g$ be an essential sphere. Up to isotopy, a link $L$ in $M_g$ crosses the sphere at say $n$ points. As $\pi(1_n-P_n)=0\in TS_n$, we can replace $1_n$ with $P_n$ without changing the value in the reduced skein module. But this has the effect of removing all the intersection points. 
This proves that the natural map $\Sred(M_g\setminus S^2)\to \Sred(M_g)$ is an isomorphism. By removing sufficiently many essential sphere, the manifold reduces to an union of balls for which the result is already known. \end{proof}

Let $H$ be a handlebody. The manifold $H\cup\overline{H}$ is homeomorphic to a connected sum of copies of $S^2\times S^1$. Hence, by the preceding corollary the gluing map $\langle\cdot,\cdot\rangle_{H,\overline{H}}:\Sred(H)\times \Sred(\overline{H})\to \Sred(H\cup \overline{H})$ can be viewed as a Hermitian form on $\Sred(H)$. 

\begin{proposition}\label{sesqui}
Let $(\Gamma,S_\Gamma)$ be a banded trivalent graph and set $H=S_\Gamma\times [0,1]$. The skein expansions $\Gamma_c$ are orthogonal with respect to the Hermitian form and satisfy
$$\langle \Gamma_c,\Gamma_c\rangle =\frac{\prod\limits_{v\in V(\Gamma)}\langle c(e_v),c(e_v'),c(e_v'')\rangle}{\prod\limits_{e\in E(\Gamma)}\langle c(e)\rangle}.$$
In this formula $e_v,e_v',e_v''$ are the three edges incident to a vertex $v$ and $\langle a,b,c\rangle$ is the skein expansion in $B^3$ of a standard graph $\Theta$ with colors $a,b,c$. We have also set $\langle n\rangle=\tr f_n=(-1)^n[n+1]$.
\end{proposition}
This proposition shows the linear independence of the family $(\Gamma_c)$ provided that one has $\langle \Gamma_c,\Gamma_c\rangle\ne 0$. 
Using induction and the formulas of Figure \ref{fusion}, we can show the following formula whose proof can be found in \cite{MasbaumVogel}.
\begin{lemma}
For any $r$-admissible triple $a,b,c$ with corresponding internal colors $i,j,k$ we have 
$$\langle a,b,c\rangle =(-1)^{i+j+k}\frac{[i+j+k+1]![i]![j]![k]!}{[a]![b]![c]!}$$
where we have set $[n]!=[n][n-1]\cdots [1]$. In particular, $\langle a,b,c\rangle\ne 0$. 
\end{lemma}

\begin{proof}(Of Proposition \ref{sesqui})
Let $c$ and $c'$ be two $r$-admissible colorings and let us compute $\langle \Gamma_c,\Gamma_{c'}\rangle\in \Sred(H\cup\overline{H})$. Fix an edge $e$ of $\Gamma$ and consider the disc $D_e$ dual to $e$ in $H$ and $\overline{D}_e$ the same disc in $\overline{H}$. Their union form an essential sphere $S^2$ in $H\cup\overline{H}$. The union $\Gamma_c\cup \Gamma_{c'}$ cut this disc at $c(e)+c'(e)$ points and contains the juxtaposition of $f_{c(e)}$ and $f_{c(e')}$. 

Let us compute $f_a\otimes f_b$ in $TS_{a+b}$. As $1_{a+b}=P_{a+b}$ in $TS_{a+b}$ by Lemma \ref{annulationspherique}, we compute instead $(f_a\otimes f_b) P_{a+b}$. Recall that one can write $P_{a+b}$ as a linear combination of simple tangles without any component going from one side to the other. A the same time, if any strand of $f_a$ (resp. $f_b$) goes back to $f_a$ (resp. $f_b$), then we get $0$. The only way to have some non-zero term is to have $a=b$ and all the strands of $f_a$ are connected to strands of $f_b$. Hence, we have the identity shown in Figure \ref{retour}.
\begin{figure}[htbp]
\begin{center}
 \def\svgwidth{10cm}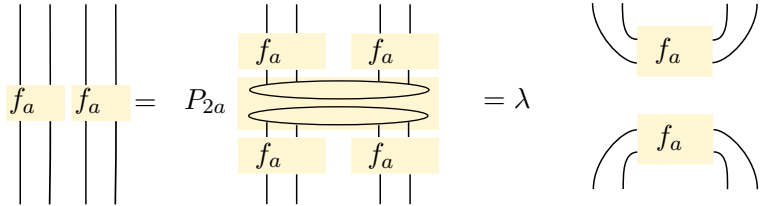
\caption{Disconnecting formula}\label{retour}
\end{center}
\end{figure}

By closing it at both sides, we get $\langle a\rangle=\lambda\langle a\rangle^2$ hence $\lambda=\frac{1}{\langle a\rangle}$. This proves orthogonality and shows that up to the $\lambda$ factors, $\langle \Gamma_c,\Gamma_c\rangle$ reduces to a disjoint union of Theta graphs. The result follows. 
\end{proof}

Let us notice that the vector space $\Sred(H)$ we just defined is indeed the same as the vector space $V_{2r}(\Sigma)$ of \cite{bhmv}. In particular, we get the following (Verlinde) formula for its dimension where $g$ denotes the genus of $\Sigma$:
$$\dim \Sred(H)=\Big(\frac{r}{2}\Big)^{g-1}\sum_{j=1}^{r-1} \Big(\sin \frac{\pi j}{r}\Big)^{2-2g}.$$

\subsection{Linking form on handlebodies}
Let $H$ be a handlebody embedded in $S^3$ in such a way that its complement $H'$ is also a handlebody. Then the gluing map $\langle\cdot,\cdot\rangle_{H,H'}:\Sred(H)\times \Sred(H')\to \Sred(S^3)\simeq K$ can be viewed as a bilinear form.  

\begin{proposition}\label{linking}
The bilinear form $\langle \cdot,\cdot\rangle_{H,H'}$ is non-degenerate. 
\end{proposition}
\begin{proof}
We first prove the case when $H$ has genus 1, the general case will follow easily. 
If $S^3=H\cup H'$ with $H,H'$ of genus 1, then they form a Hopf link as in the left hand side of Figure \ref{hopf}. Take  standard banded links $\Gamma$ and $\Gamma'$ in $H$ and $H'$ such that the handlebodies retract on them. The vectors $(\Gamma_i)$ and $(\Gamma'_j)$ for $i,j\in \{0,\ldots,r-2\}$ form a basis of $\Sred(H)$ and $\Sred(H')$ respectively so that we have to prove that the matrix $\Pi_{ij}=\langle \Gamma_i,\Gamma_j\rangle_{H,H'}$ is non-degenerate. This will follows from an explicit computation that we explain now.

\begin{figure}[htbp]
\begin{center}
 \def\svgwidth{10cm}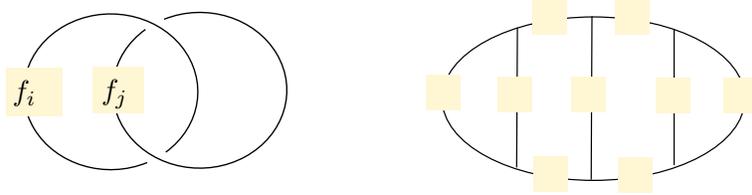
\caption{Hopf link and a generalised Theta graph}\label{hopf}
\end{center}
\end{figure}

From Lemma \ref{cerclage}, we know that encircling $f_i$ with $z=\hat{f_1}$ multiplies it by $-A^{2i+2}-A^{-2i-2}$. As $\hat{f}_j=S_j(z)$, encircling $f_i$ with $f_j$ will multiply it with $S_j(-A^{2i+2}-A^{-2i-2})=(-1)^j\frac{A^{2(j+1)(i+1)}-A^{-2(i+1)(j+1)}}{A^{2i+2}-A^{-2i-2}}$. This gives finally $\Pi_{ij}=(-1)^{i+j}[(i+1)(j+1)]$. This matrix is a kind of discrete Fourier transform: in particular, one computes 
$$\Pi^2=\frac{-2r}{(A^2-A^{-2})^2}\id.$$
This shows that $\Pi$ is invertible, concluding the case $g=1$. 

Let $S$ be a standard annulus. The preceding computation shows that for all $i\in \{0,\ldots, r-2\}$, there exists $t_i\in \Sred(S)$ such that encircling $f_j$ with $t_i$ yields $\delta_{ij}f_j$.  Let $H$ be a standard genus $g$ handlebody in $S^3$ and $\Gamma$ be a banded trivalent graph obtained by duplicating the central edge of a theta graph as in the right hand side of Figure \ref{hopf}. Denote by $\delta_e\subset H'$ the boundary of the dual disc of $e$. For any $r$-admissible coloring of $\Gamma$, consider the element $\Delta_c$ of $\Sred(H')$ obtained by inserting $t_{c(e)}$ along $\delta_e$ for all $e\in E(\Gamma)$. We observe that $\langle \Gamma_{c'},\Delta_c\rangle_{H,H'}$ is zero unless $c=c'$ and in that case it is equal to the evaluation of $\Gamma_c$ in $S^3$ which is non zero (as a product of Theta coefficients). We conclude that any combination of $\Gamma_c$ which is in the radical of the form $\langle \cdot,\cdot\rangle$ must vanish, hence the form $\langle\cdot,\cdot\rangle_{H,H'}$ is non-degenerate. 
\end{proof}

\subsection{The reduced skein module of a surface}
Let $H$ be a handlebody in $S^3$ such that its complement $H'$ is also a handlebody. We denote by $\Sigma$ their common boundary.  We take $i:\Sigma\times [0,1]\to S^3$ an embedding such that Im$(i)\cap H=\Sigma$. 

The usual gluing map defines an action of the algebra $\Sred(\Sigma)$ on $\Sred(H)$.  We will denote by $\Phi(L)$ the action of $L\in \Sred(\Sigma)$ : this is called the {\it curve operator} associated to $L$. The main result of this article is the following. 

\begin{theorem}
The natural map $\Phi:\Sred(\Sigma)\to \en(\Sred(H))$ is an isomorphism of algebras. 
\end{theorem}

\begin{proof}

Let $D$ be a disc embedded in $\Sigma$ and set $H_1=\Phi((\Sigma\setminus D) \times [0,1])$ and $H_2$ its complement in $S^3$. We can find trivalent banded graphs $\Gamma_1\subset H_1$ and $\Gamma_2\subset H_2$ on which they retract by deformation, and we can suppose that there is a unique edge $e\in E(\Gamma_2)$ which intersect $D$ in one point, see Figure \ref{toretroue}. 

\begin{figure}[htbp]
\begin{center}
 \def\svgwidth{7cm}
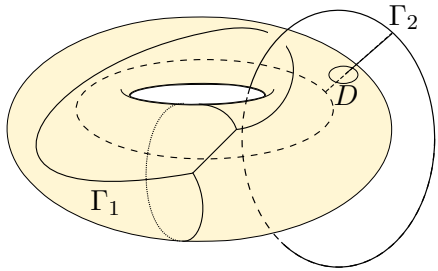
\caption{The graphs $\Gamma_1$ and $\Gamma_2$}\label{toretroue}
\end{center}
\end{figure}

As any graph in $\Sigma\times[0,1]$ can be made disjoint from $D\times[0,1]$ up to isotopy, the natural map $\Psi:\Sred(H_1)\to \Sred(\Sigma)$ is surjective. 
Denote by $\gamma$ the boundary of the disc $D$. It may be pushed into $H_1$ or $H_2$, defining two endomorphisms $T_1,T_2$ acting respectively on $\Sred(H_1)$ and $\Sred(H_2)$. These endomorphisms are adjoint with respect to the linking form of $H_1$ and $H_2$ i.e. 
\begin{equation}\label{adjoint}\langle T_1(x),y\rangle_{H_1,H_2}=\langle x,T_2(y)\rangle_{H_1,H_2}\text{ for all }x\in \Sred(H_1)\text{ and }y\in \Sred(H_2).
\end{equation}
 
As $\gamma$ bounds a disc in $\Sigma$, we see that $\Psi(T_1(x))=-(A^2+A^{-2})\Psi(x)$ for all $x\in \Sred(H_1)$. Thanks to Formula \ref{adjoint}, the subspace $\im(T_1+(A^2+A^{-2})\id)$ is the orthogonal of $\ker(T_2+(A^2+A^{-2})\id)$ with respect to the linking form. But as $\gamma$ encircle the edge $e$ of $\Gamma_2$ we have $T_2((\Gamma_2)_c)=(-A^{2c(e)+2}-A^{-2c(e)-2})(\Gamma_2)_c$. It follows that $\ker(T_2+(A^2+A^{-2})\id)$ is generated by graphs $\Gamma_2$ with a $r$-admissible coloring $c$ satisfying $c(e)=0$. Such colorings are in bijection with colorings of $\Gamma_2\setminus e$ which is the disjoint union of two graphs $\Gamma,\Gamma'$ on which $H$ and $H'=S^3\setminus H$ retract respectively. 
We conclude that we have the inequality $\dim \Sred(\Sigma)\le \dim \Sred(H\amalg H')=\dim \Sred(H)\dim \Sred(H')=\dim \en \Sred(H)$. 

Reciprocally, the non-degeneracy of $\langle \cdot,\cdot\rangle_{H_1,H_2}$ tells us the following. Given $r$-admissible colorings $c,c'$ on $\Gamma$ and $\Gamma'$, there exists $x\in \Sred(H_1)$ such that $\langle x,\Gamma_d\cup \Gamma_{d'}\rangle_{H_1,H_2}=\delta_{cc'}\delta_{dd'}$ for all $d,d'$. But $\langle x,\Gamma_d\cup\Gamma_{d'}\rangle_{H_1,H_2}=\langle \Psi(x)\Gamma_d,\Gamma_{d'}\rangle_{H,H'}$. As $\langle \cdot,\cdot\rangle_{H,H'}$ is also non-degenerate, we conclude that $\Psi$ is surjective and the theorem follows. 
\end{proof}

Let us derive from this theorem a formula for the action of Dehn twists. Let $f\in \mcg(H)$ be a diffeomorphism of $H$ preserving $\partial H=\Sigma$. It induces a diffeomorphism of $\Sigma$ that we denote by the same letter. By naturality of the action of $\Sred(\Sigma)$ on $\Sred(H)$,  we have the following relation: 
$$\Phi(f(x))(f(y))=f(\Phi(x)(y)),\quad \forall x\in\Sred(\Sigma), \forall y \in \Sred(H).$$
Stated differently, this gives $\Phi(f(x))=f\circ \Phi(x)\circ f^{-1}$, and hence $\rho(f)=f$. This shows that the representation $\rho$ restricted to $\mcg(H)$ coincides with the natural action of $\mcg(H)$ on $\Sred(H)$. In particular, it is linear, not only projective. 

Suppose that $\Gamma$ is a trivalent banded graph embedded in $H$ as usual. Then, for any edge $e$ of $\Gamma$ and dual disc $D_e$, the Dehn twist along $D_e$ is an element of $\mcg(H)$ which restricts on the boundary to the Dehn twist on $\partial D_e$. 
Using the formula for $z_n$ in Lemma \ref{cerclage}, we get immediately the following proposition. 

\begin{proposition}
For any edge $e$ of $\Gamma$, the Dehn twist $t_e$ along the dual disc $D_e$ acts diagonally on the standard basis of $\Sred(H)$. More precisely, we have $$t_e(\Gamma_c)=(-1)^{c(e)}A^{c(e)(c(e)+2)}\Gamma_c.$$
\end{proposition}

\section{Further topics}

\subsection{Explicit formulas}\label{sectionexplicit}
Suppose that $S^3=H\cup H'$ where $H$ and $H'$ are two handlebodies with common boundary $\Sigma$. We showed in the last section how the group $\mcg(H)$ acts on $\Sred(H)$. Roberts approach of the quantum representation is to use the duality between $\Sred(H)$ and $\Sred(H')$ given by the linking form to get an action of $\mcg(H')$ on $\Sred(H)$, see \cite{skeinroberts}. As $\mcg(H)$ and $\mcg(H')$ generate the mapping class group (see \cite{fm}, Chap. 4), this is enough to compute the representation. Unfortunately, the formulas for the linking form are intractable in practice as soon as the genus of $H$ is greater than $1$. 

We explain here a more efficient way for computing the quantum representation. 
Suppose that $(\Gamma,S_\Gamma)$ is a banded trivalent graph such that $H=S_\Gamma\times [0,1]$ and such that $S_\Gamma$ has genus $0$ and $\Gamma$ has no disconnecting edge. This implies that each boundary component $\gamma$ of $S_\Gamma$ passes at most once along each edge of $\Gamma$ (under the retraction of $S_\Gamma$ on $\Gamma$). We will identify the component $\gamma$ with the corresponding subgraph of $\Gamma$.
Using simple fusion rules of Lemma \ref{fusion}, we observe that the endomorphism $\Phi(\gamma)$ has a simple ``tridiagonal" expression in the standard basis. 
Formally, we have the following lemma:
\begin{lemma}\label{explicit}
For any component $\gamma$ of $\partial S_\Gamma$ as above, one can write for any $r$-admissible coloring $c:E(\Gamma)\to\{0,\ldots,r-2\}$
$$\Phi(\gamma)(\Gamma_c)=\sum_{\xi:E(\gamma)\to \{\pm 1\}} F(c,\xi)\Gamma_{c+\xi}$$
where $F(c,\xi)\ne 0$ if and only if $c+\xi$ is $r$-admissible. 
\end{lemma}
The coefficients $F(c,\xi)$ are cumbersome quotients of quantum integers. Let us look at the example where $H$ is a handlebody of genus 2 which retracts on a Theta graph $\Gamma$. Denote by $\Gamma_{a,b,c}$ the skein expansion of $\Gamma$ with colors $a,b,c$ and by $\gamma$ the curve passing along the edges colored by $a$ and $b$ as in Figure \ref{Theta}.

\begin{figure}[htbp]
\begin{center}
 \def\svgwidth{12cm}
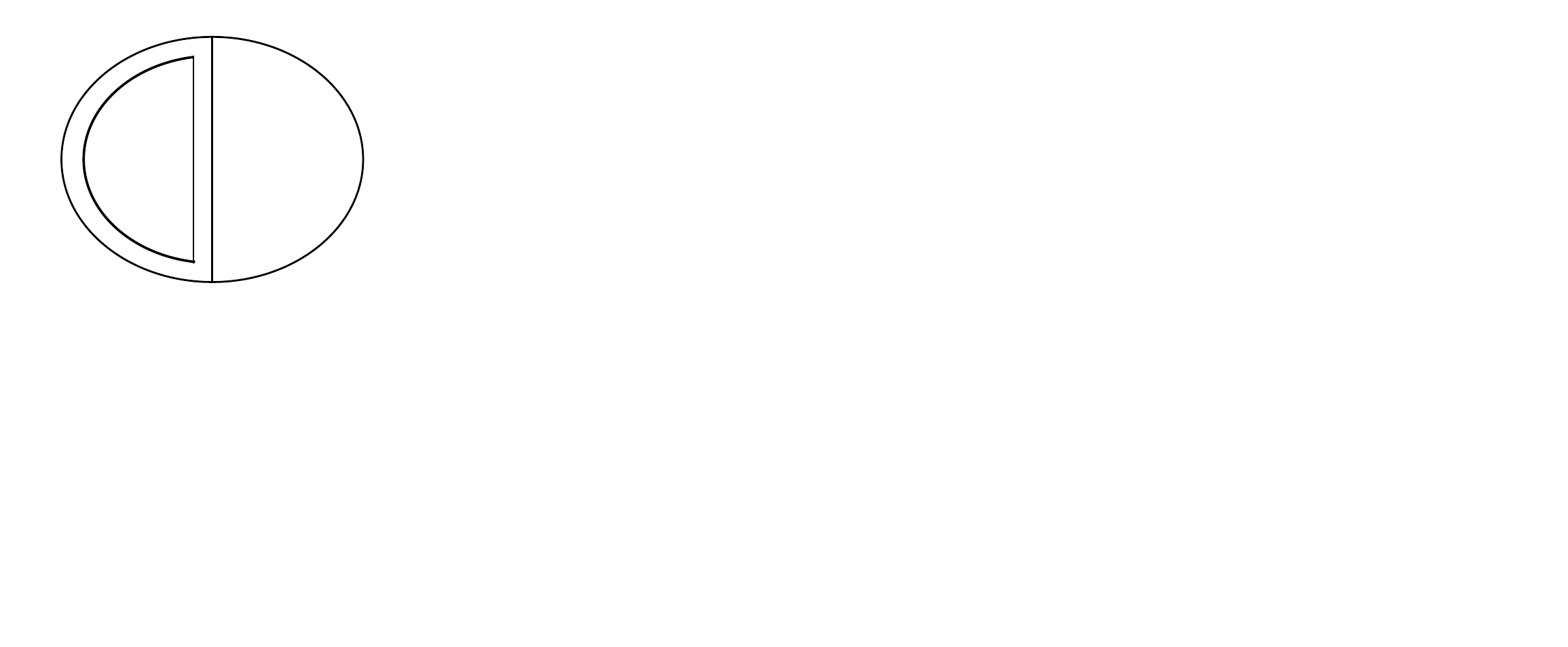
\caption{Curve operator on a genus $2$ surface}\label{Theta}
\end{center}
\end{figure}
Reducing the graph using fusion rules, we get the following formula for $\Phi(\gamma)$. 
\begin{equation*}
\begin{aligned}
\Phi(\gamma)(\Gamma_{a,b,c})=&\Gamma_{a+1,b+1,c}-\frac{[(c+a-b)/2]^2}{[a][a+1]}\Gamma_{a-1,b+1,c}-\frac{[(c+b-a)/2]^2}{[b][b+1]}\Gamma_{a+1,b-1,c}\\
&+\frac{[(a+b+c)/2+1]^2[(a+b-c)/2]^2}{[a][a+1][b][b+1]}\Gamma_{a-1,b-1,c}.
\end{aligned}
\end{equation*}
If $\gamma$ were bounding a dual disc $D_e$, we would have instead $\Phi(\gamma)(\Gamma_c)=-(A^{2c(e)+2}+A^{-2c(e)+2})\Gamma_c$ and $t_\gamma (\Gamma_c)=(-1)^{c(e)}A^{c(e)(c(e)+2)}$. This shows that given a polynomial $Q\in K[X]$ satisfying $Q(-A^{2n+2}-A^{-2n-2})=(-1)^nA^{n(n+2)}$ for all $n\in \{0,\ldots r-2\}$ we have the practical formula: 
$$\rho(t_\gamma)=Q(\Phi(\gamma)).$$
As all Dehn twists are conjugate by the mapping class group to a Dehn twist bounding a dual disc, the above formula is always true and can be used to compute the image of Dehn twists from the expression of Lemma \ref{explicit}. This algorithm is already implemented in \cite{tqft.gp}. The tridiagonal form of the curve operator $\Phi(\gamma)$ generalises to all elements of $\Sred(\Sigma)$. This gives a reinterpretation of the curve operators as Toeplitz operators, a key ingredient for understanding the semi-classical properties of the quantum representations, see \cite{mp,detcherry}. 

\subsection{Irreducibility}

%

Let us use the formulas of the last section to reprove the following theorem of Roberts (see \cite{roberts}).

\begin{proposition}
Let $r$ be an odd prime. The quantum projective representation of $\mcg(\Sigma)$ on $\Sred(H)$ is irreducible. 
\end{proposition}
\begin{proof}
We start with the following lemma. 
\begin{lemma} 
For all $n,m\in \{0,\ldots,r-2\}$, $(-1)^n A^{n(n+2)}=(-1)^m A^{m(m+2)}$ if and only if $n=m$.
\end{lemma}
\begin{proof}
Writing $-1=A^{2r}$, the equality is equivalent to $2nr+n(n+2)=2mr+m(m+2)$ modulo $4r$ or $4r|(n-m)(2r+n+m+2)$. If $n\ne m$, $r$ should divide $2r+n+m+2$. As $2\le n+m+2\le 2r-2$, we have $n+m+2=r$ and the condition is $4r|(n-m)3r$. Hence $n$ and $m$ have the same parity which contradicts the previous equality. 
\end{proof}
As a consequence of this lemma, we find that $\rho(t_\gamma)$ and $\Phi(\gamma)$ have exactly the same eigenspaces, hence an element $\Psi\in \en(\Sred(H))$ which commutes with $\rho(t_\gamma)$ for all simple curves $\gamma$ also commutes with $\Phi(\gamma)$. Consider such an endomorphism. 
As $\Psi$ should preserve the eigenspaces of $\Phi(\gamma_e)$ for all edges $e$, it must be diagonal in the basis $(\Gamma_c)$. Hence we can write $\Phi(\Gamma_c)=\lambda_c\Gamma_c$. Using the formula of Lemma \ref{explicit}, the commutation of $\Psi$ and $\Phi(\gamma)$ for $\gamma$ a component of $\partial S_\Gamma$ reads $F(c,\xi)(\lambda_c-\lambda_{c+\xi})=0$ for all $\xi:E(\gamma)\to \{\pm 1\}$ hence $\lambda_c=\lambda_{c+\xi}$ if both $c$ and $c+\xi$ are $r$-admissible. 

It remains to show that this implies that $\lambda_c$ is constant. It is easy to show by induction for a generalised Theta graph: we left it as an exercise for the reader. 
\end{proof}

\subsection{Integral structure}
Let $r$ be an odd prime and set as usual $K=\Q[A]/\phi_{4r}(A)$ the cyclotomic field of order $4r$. We denote by $\O$ the ring of integers of $K$. We will make use of the following lemma.
\begin{lemma}
The quantum integers $[1],[2],\ldots,[r-1]$ are units of $\O$. 
\end{lemma}
\begin{proof}
Write $q=A^4$: it satisfies $q^r=1$ and we have $[n]=A^{2-2n}\frac{q^n-1}{q-1}$. Hence it is sufficient to prove that $1+q+\cdots+q^{n-1}$ is a unit in the ring $\Z[q]/(1+q+\cdots+q^{r-1})$ for $n<r$. Hence we want to show that the greatest common divisor of $1+\cdots+q^{r-1}$ and $1+q+\cdots+q^{n-1}$ in $\Z[q]$ is $1$. When applying the Euclidean division algorithm, we find that the situation is similar to the computation of the greatest common divisor of $r$ and $n$ in $\Z$. As these numbers are coprime by hypothesis, the conclusion follows.
%
\end{proof}

This lemma and a simple induction implies that the Jones-Wenzl idempotents $f_0,\ldots,f_{r-1}$ have coefficients in $\O$, hence one can define an integral version $\Sred_\o(M)$ of the reduced skein module simply by taking the ground ring $R=\O$ instead of $K$. It clearly satisfies $\Sred_\o(M)\otimes K=\Sred(M)$. 


\begin{proposition}\label{finitude}
For any manifold $M$, $\Sred_\o(M)$ is a finitely generated $\O$-module, free if $M$ is a handlebody. 
\end{proposition}
\begin{proof}
One can assume that $M$ is connected and that there exists a handlebody $H\subset M$ such that any banded link in $M$ lives in $H$ up to isotopy. Hence the map $\Sred_\o(H)\to \Sred_\o(M)$ is surjective and it is sufficient to deal with the case when $M$ is a handlebody.

Pick a banded trivalent graph $\Gamma\subset S\subset M$ such that $M$ retracts on $S$ which retracts on $\Gamma$. As in Theorem \ref{structure_skein}, by solving the crossings of a diagram and removing trivial components, we observe that $\Sred_\o(\Sigma)$ is generated as an $\O$-module by simple links in $S$. We can then reproduce the proof of Lemma \ref{reduction} thanks to the integral properties of the idempotents. It follows that the standard basis $(\Gamma_c)$ is also a basis as an $\O$-module.  
\end{proof}

This proves that $\Sred_\o(\Sigma)$ is an order, preserved by the quantum representation. If the map $\Phi_\o:\Sred_\o(\Sigma)\to \en_\o \Sred_\o(H)$ were an isomorphism,  it would follow that the quantum representation preserves the lattice $\Sred_\o(H)\subset \Sred(H)$ (recall that by lattice we mean a finitely generated $\O$-submodule of $\Sred(H)$, hence without torsion). Although this is not the case, the experts will immediately see that the representation must preserve a lattice. For convenience, we provide a proof below. 

\begin{proposition}
Suppose that $\Sigma$ has genus $g\ge 3$. Then there exists a lattice $\Lambda\in \Sred(H)$ which is (projectively) preserved by $\mcg(\Sigma)$. 
\end{proposition}
\begin{proof}
We suppose that $g\ge 3$ to simplify the proof but one can adapt it so that it works for all $g\ge 1$. Recall from \cite{fm} that $\mcg(\Sigma)$ is a perfect group as $g\ge 3$ and denote by $\mcgu(\Sigma)$ its universal central extension (by $\Z$). 
By the universal property, there is a lift $\tilde{\rho}:\mcgu(\Sigma)\to \GL(\Sred(H))$. Composing with the determinant map and observing that $\mcgu(\Sigma)$ is also perfect, we can replace $\rho$ with the following representation:
$$\tilde{\rho}:\mcgu(\Sigma)\to \SL(\Sred(H)).$$

The cokernel of the map $\Phi_\o$ is a finitely generated torsion $\O$-module. Hence one can write $\coker \Phi_\o=\bigoplus_{i=1}^m R/P_i^{\alpha_i}$ where $P_1,\ldots,P_m$ are prime ideals in $\O$ and the $\alpha_1,\ldots,\alpha_m$ are positive integers. 
Fix $i\in \{1,\ldots, m\}$, and denote by $K_i$ the $P_i$-adic completion of $K$ and $\O_i$ its valuation ring. We also denote by $v_i:K_i^\times\to \Z$ the valuation and chose a uniformizer $\pi_i\in \O$ such that $v_i(\pi_i)=1$. 
Fixing a basis of $\Sred_\o(H)$, we may view $\tilde{\rho}$ as a representation in $\SL_n(K_i)$. If we set $\Sred_{\O_i}(\Sigma)=\Sred_\o(\Sigma)\otimes \O_i$ and denote by $\Phi_i:\Sred_{\O_i}(\Sigma)\to M_n(\O_i)$ the natural map, our assumptions imply that 
$$\pi_i^{\alpha_i}M_n(\O_i)\subset \Phi_i(\Sred_{\O_i}(\Sigma))\subset M_n(\O_i).$$

As $\tilde{\rho}$ preserves $\Phi_i(\Sred_{\O_i}(\Sigma))$ by conjugation, we get 

\begin{equation}\label{conj}
\tilde{\rho}(f) M_n(\O_i) \tilde{\rho}(f)^{-1}\subset \frac{1}{\pi_i^{\alpha_i}}M_n(\O_i)\quad\text{ for any }f\in\mcgu(\Sigma).
\end{equation}
As $\O_i$ is a principal ideal domain, the invariant factors theorem implies that one can write $\tilde{\rho}(f)=K_1AK_2$ where $K_1,K_2\in \SL_n(\O_i)$ and $A$ is a diagonal matrix with entries $a_1,\ldots,a_n\in K_i$. As $M_n(\O_i)$ is invariant by conjugation by $K_1$ and $K_2$, we get the same equation as \eqref{conj} if we substitute $\tilde{\rho}(f)$ with $A$. 

This equation implies $v_i(a_ra_s^{-1})\ge -\alpha_i$ for all $r,s\in \{1,\ldots,n\}$. Hence we have $|v_i(a_r)-v_i(a_s)|\le \alpha_i$ for all $r,s$ and $\sum v_i(a_r)=0$ because $\det A=1$. This shows $v_i(a_r)\ge \frac{1-n}{n}\alpha_i\ge -\alpha_i$. Recalling that $\tilde{\rho}(f)=K_1AK_2$, we get finally $\tilde{\rho}(f)\in \frac{1}{\pi_i^{\alpha_i}}M_n(\O_i)$. Writing $D=\prod_{i=1}^m \pi_i^{\alpha_i}\in \O$, we get $\tilde{\rho}(f)\in \frac{1}{D}M_n(\O)$. 

Consider $\Lambda$ the $\O$-submodule of $\Sred(H)$ generated by $\tilde{\rho}(f)\Sred_\o(H)$ for $f\in\mcgu(\Sigma)$. We found $\Sred_\o(H)\subset \Lambda\subset \frac{1}{D}\Sred_\o(H)$ hence $\Lambda$ is a lattice, obviously preserved by $\tilde{\rho}$.
\end{proof}
\begin{remark} One can get information on the denominator $D$ provided by the proof by the following observation. Take $P$ a prime ideal in $\O$ with residue field $k$. Tensoring the integral reduced skein module with $k$ gives yet a new version of the theory, in particular a map $\Phi_k:\Sred_k(\Sigma)\to \en \Sred_k(H)$. The proof that $\Phi_k$ is an isomorphism works provided that the $\Pi$-matrix of Proposition \ref{linking} is invertible over $k$. Hence $D$ can be chosen to be a power of $-2r/(A^2-A^{-2})^2$. This simplifies further once we notice that $r=(A^2-A^{-2})^{r-1}$ modulo a unit of $\O$. 
\end{remark}
One can find a lattice which is free as a $\O$-module and find an explicit basis, see \cite{gm}. This integral structure gives the most interesting applications of the quantum representations, see \cite{gm2, mr,ks}.

\subsection{Marked points}
The theory extends to the case where the surface is endowed with colored marked points. A collection of marked points on a surface $\Sigma$ indexed by $I$ is a family of embeddings $\phi_i:[0,1]^2\to \Sigma$  for $i\in I$ with disjoint images. 
The colors of the marked points are given by a map $c:I\to \{0,\ldots,r-2\}$ which correspond the the system of arcs 
$$P=\bigcup_{i\in I} \phi_i(P_{c(i)}).$$

We define the element $f_{c}\in \Sred(\Sigma,P)$ by the formula $f_c=\bigotimes_{i\in I} \phi_i(f_{c(i)})$. This is an idempotent in the algebra $\Sred(\Sigma,P)$ where the product is given as usual by stacking. We then define the subalgebra $$\Sred(\Sigma,c)=f_c \Sred(\Sigma,P)f_c\subset\Sred(\Sigma,P).$$

Given a handlebody $H$ with boundary $\Sigma$, this algebra acts on $\Sred(H,c)=\im (\Phi(f_c))\subset \Sred(H,P)$ and we can show the following proposition. 

\begin{proposition} The natural action of $\Sred(\Sigma,c)$ on $\Sred(H,c)$ induces an isomorphism $\Sred(\Sigma,c)\simeq \en \Sred(H,c)$. In particular, the mapping class group $\mcg(\Sigma,P)$ which is the identity on $\phi_i([0,1]^2)$ acts on $\Sred(\Sigma,c)$ by automorphisms, which produces a projective representation $$\rho_c:\mcg(\Sigma,P)\to \PGL( \Sred(H,c)).$$
\end{proposition}

Moreover, one can show that the reduced colored skein module $\Sred(H,c)$ has a basis given by skein expansions of graphs $\Gamma$ embedded in $H$ with univalent vertices lying on the marked points. The $r$-admissible colorings of these graphs are colorings which coincide with $c$ on the boundary points. 

For instance, in the case of the once-punctured torus, one can consider the banded graph shown in Figure \ref{graph}. If one gives the color $c$ to the marked point, the color $a$ of the remaining edge should satisfy $c\le 2a \le 2r-4-c$. Hence, most of what we showed in these notes generalise to the case of marked points.

\end{document}